\documentclass[12pt]{amsart}
\usepackage{xcolor}
\usepackage[draft]{todonotes} 
\usepackage[english]{babel}
\usepackage{geometry}
\usepackage[T1]{fontenc}
\usepackage[latin1]{inputenc}
\usepackage{amsmath,amssymb,mathrsfs,amsthm, tikz-cd,mathrsfs}
\usepackage{soul}
\usepackage{epsfig}

\usepackage{hyperref} 
\hypersetup{
    colorlinks=true,
    linkcolor=blue,
    urlcolor=red,
    pdftitle={},
    }

\usepackage{graphicx}
\usepackage{xypic}
\usepackage{enumitem}
\usepackage{datetime}
\usepackage{xcolor}
\usepackage{fancyhdr}

\newtheorem{theorem}{Theorem}[section]

\newtheorem{lemma}[theorem]{Lemma}
\newtheorem{conjecture}[theorem]{Conjecture}

\newtheorem{question}[theorem]{Question}
\newtheorem{remark}[theorem]{Remark}

\newcommand{\C}{\mathbb{C}}

\newcommand{\Z}{\mathbb{Z}}

\newcommand{\p}{\mathbb{P}}

\newcommand{\Q}{\mathbb{Q}}
\newcommand{\N}{\mathbb{N}}

\title{Toward a generalization of  Lehmer's problem to  adelic curves}

\author{Mounir Hajli}

\address{School of Science, Westlake University}

\email{hajli@westlake.edu.cn}

\begin{document}

\maketitle

\begin{abstract} 
In this short note, we investigate the generalization of Lehmer's problem to finitely generated fields over $\Q$.
\end{abstract}

\begin{center}
\end{center}

\begin{center}


\end{center}

\tableofcontents

\section{Introduction}

Moriwaki \cite{Moriwaki1} 
has developed an Arakelov height theory for varieties over a finitely generated extension of a number field. He applied his theory   to the study of Bogomolov problem over finitely generated fields over $\Q$ and obtained generalizations of results due to Ullmo \cite{Ullmo1998} and Zhang \cite{Zhang1998}. As a corollary, he recovered the original Raynaud's theorem (Manin-Mumford's conjecture).

Using the height machinery developed in \cite{Moriwaki1}, Vojta \cite{Vojta-2021} extended the classic Roth's
theorem to finitely generated fields.       Burgos, Philippon and Sombra \cite{Burgos2016} have showed that the height of a variety over a finitely generated extension of $\Q$ can be written as an integral of local heights over the set of places of the field.  Using their previous work on toric varieties in 
\cite{Burgos2}, they gave a combinatorial formula for the height  of  toric varieties over finitely generated fields over  $\Q$.  
\

Chen and Moriwaki \cite{Huayi-Atsushi-Adelic} have introduced the theory of adelic curves. For example, a number field can be viewed as an example of adelic curve. They developed an arithmetic intersection theory on these curves. For instance, one can define a height function on the set of points of adelic curves. The Arakelov height theory of \cite{Moriwaki1} is then seen as an arithmetic intersection theory in the sense of \cite{Huayi-Atsushi-Adelic}.\\

We briefly recall the classical Lehmer's problem (for details, see~\S\ref{sec2}). It can be formulated as follows: There exists a constant \( c > 0 \) such that for every non zero algebraic number \(\alpha\) of degree \( d \) over \(\mathbb{Q}\) that is not a root of unity, we have

\[
h(\alpha) \geq \frac{c}{d}.
\]

Lehmer's question has not yet been settled. The best result in the direction of solving Lehmer's problem is  Dobrowolski's result \cite{Dobrowolski} (see \eqref{Dobrowolski}).\\

Our  work is  an attempt to generalize Lehmer's problem to adelic curves.  Let $S=(K,(\Omega,\mathcal A,\nu),\phi)$ be a proper adelic curve and $K^{ac}$ an algebraic 
closure of $K$ (see \S~\ref{AC}). A natural question arises:  Is there a constant $c>0$  such for every $\alpha\in K^{\mathrm{ac}}$ with $h_S(\alpha)>0$, we have  \[
h_S(\alpha)\geq \frac{c}{[K(\alpha):K]}?
\]
where 
$h_S$ is the standard height function (see Paragraph \ref{2.1}).\\

In this paper, we show that Lehmer's problem generalizes to a natural adelic curve defined on the field of rational functions with coefficients in 
$\Q$. Let us denote by $S$ the adelic curve on $\Q(T)$ defined in Paragraph \ref{QT}.    Our first main result is  a generalization of Kronecker's theorem to 
$S$. More precisely, we prove the following.

\begin{theorem}[\rm{A Kronecker's theorem}]

\[
 \{ \alpha \in (\overline{\Q(T)})^\ast \mid h_S(\alpha)=0\}=\{\text{roots of unity in} \ \overline{\Q}\}.
\]

\end{theorem}

Our second main result can be stated as follows.
\begin{theorem}
Assuming  Lehmer's conjecture (see Conjecture \ref{Conjecture 1.1.} below), then for every $\alpha$ in an algebraic closure of $\Q(T)$ which is not a root of unity, we have
\[
h_S(\alpha)\geq \frac{c}{\deg(\alpha)}\tag{Lehmer's problem on $S$},
\]
where  $\deg (\alpha)=[\Q(T)(\alpha):\Q(T)]$ and $c$ is an absolute positive constant.
\end{theorem}
(see Theorem \ref{Lehmer}).\\

An interesting question is to explore the possibility of adapting Dobrowolski's methods in our setting, which will be addressed in forthcoming papers.\\

Let us note that generalizations of Lehmer's problem to higher-dimensional varieties have been explored by Amoroso and David (cf.~\cite{AD1999, AD2000, AD2003}).  As a continuation of this work, we propose two conjectures, namely Conjectures \ref{c1} and \ref{c2} below, which we plan to investigate in a forthcoming paper.  Let $n$ and $d$ be two positive integers.  
\begin{conjecture}\label{c1}
Let \( \boldsymbol{\alpha} = (\alpha_1, \ldots, \alpha_n) \in \mathbb{G}_m^n\left(\overline{\mathbb{Q}(T_1,\ldots,T_d)}\right) \), such that \( \alpha_1, \ldots, \alpha_n \) are multiplicatively independent. We have
\[ h_S(\boldsymbol{\alpha}) \geq \frac{c(n)}{\delta_{{\Q(T_1,\ldots,T_d)}}(\boldsymbol{\alpha})}, \]
where $h_S$ is the height associated with the adelic curve $\Q(T_1,\ldots,T_d)$, and   $\delta_{\Q(T_1,\ldots,T_d)}$ is, by definition, $\deg F_0$ where  \( F_0 \in \mathbb{Q}(T_1,\ldots,T_d)[x_1, \ldots, x_n] \) is  a non-zero polynomial that vanishes at \( \boldsymbol{\alpha} \) and has minimal degree. 
\end{conjecture}
By Remark~\ref{rmk5.2}, it is clear that this constitutes a natural generalization of Conjecture~\ref{2.4} due to Amoroso and David.    {}{Observe that more refined versions have been developed, taking into account the possibility that $\alpha$ lies in a proper subgroup, see for instance,  \cite{DP-2007,Carrizosa-2009,Carrizosa-survey}. \\}

Let \( V \) be a hypersurface in \(\mathbb{G}_m^n\) defined over \(\mathbb{Q}(T_1, \ldots, T_d)\). Let us consider the normalized height \(\hat{h}_S(V)\) of \( V \), which can be defined using methods of \cite{Huayi-Atsushi-Adelic}.

\begin{conjecture}\label{c2}
  Assume that \( V \) is defined over \( \mathbb{Q}(T_1,\ldots,T_d) \) and that it is not an union of torsion subvarieties. Then,
\[
\hat{h}_S(V) \geq c(n) \, .
\]
where $c(n)$ is a positive constant which does not depend on $V$.
\end{conjecture}

Note that if \( V \) is defined over \(\mathbb{Q}\), then one can show that \(\widehat{h}_S(V) = m(f)\), where \( m(f) \) is the Mahler measure of an equation \( f \) defining \( V \) (cf.~\S\ref{sec2} for details). Conjecture \ref{c2} then reduces to Conjecture \ref{mahler2.3}.

\

\emph{Organization of the paper}: In \S\ref{sec2}, we provide the definition of the usual Weil height function on \(\mathbb{P}^N(\overline{\mathbb{Q}})\) and recall the classical Lehmer's problem along with some results concerning this problem. In \S\ref{sec3}, we review the definition of the canonical height on Drinfel'd modules and briefly summarize some results concerning Lehmer's problem in this context. In \S\ref{AC}, we introduce the theory of adelic curves and discuss the extension of Lehmer's problem to this setting. Finally, in \S\ref{main}, we prove our main results on Lehmer's problem on the adelic curve \(\mathbb{Q}(T)\).

\section{Lehmer's problem on  $\overline{\Q}$ }\label{sec2}

We denote by $\mathbb{Q}$ the field of rational numbers, $\mathbb{Z}$ the ring of rational integers, $K$ a number field,  $\mathcal{O}_K$ its ring of integers, and $\overline \Q$ an algebraic closure of $\Q$. We denote by $\mathbb{R}$ the field of real numbers and $\mathbb{C}$ the field of complex numbers. We denote by $\mathbb{N}$ the set of   
 natural numbers including $0$.

For each rational prime $p$, we denote by $|\cdot|_p$ the $p$-adic absolute value on $\mathbb{Q}$ such that $|p|_p = p^{-1}$; we also denote by $|\cdot|_\infty$ or simply $|\cdot|$ the standard absolute value. These form a complete set of absolute values on $\mathbb{Q}$: we identify the set $M_\mathbb{Q}$ of these absolute values with the set $\{\infty, p \mid  p \text{ prime}\}$. More generally, we denote by $M_K$ the set of places of $K$ extending those of $M_\mathbb{Q}$. Let $x\in K$,  we define the absolute value $|x|_v = |N_{K_v/\mathbb{Q}_v}(x)|_v^{1/[K_\nu:\mathbb{Q}_v]}$ for each $v \in M_K$, where 
$[K_v:\Q_v]$ is the local degree at $v$. We have the product formula
\[
\prod_{v \in M_K} |x|_v^{[K_v : \mathbb{Q}_v]} =1,
\]

for all \( x \in K\setminus\{0\} \) (cf. \cite[Chapter V]{Lang-ANT}).\\

Let $K$ be a number field, $\mathcal O_K$ be the ring of integers, and $M_K$ be the set of places.  Let $N$ be a positive integer. Let $\mathbb{P}^n$ be the projective space of dimension $N$. Let
 $P=[x_0:\cdots:x_n] \in \mathbb{P}^N(K)$. We define the {absolute logarithmic height} of the point $P$ as

\[
h(P) =  \sum_{v \in M_K}  \frac{[K_v:\Q_v]}{[K:\mathbb{Q}]}\log  \max (|x_0|_v, \ldots, |x_n|_v),
\]
which is independent of the choice of the projective coordinates by the product formula. The {}{height} $h(P)$ is independent of the choice of the base field $K$ (cf. \cite[Chapter 3]{Lang-FDG}). This defines the usual Weil height on
$\mathbb{P}^n(\overline \Q)$.\\

Let \( f \in \mathbb{C}[x_1, \ldots, x_n] \). We define its Mahler measure by setting \( m(0) = 0 \) and

\[
m(f) = \frac{1}{(2\pi)^n} \int_0^{2\pi} \cdots \int_0^{2\pi} \log |f(e^{i\theta_1}, \ldots, e^{i\theta_n})| \, d\theta_1  \cdots  d\theta_n,
\]

if \( f \neq 0 \).  Let \( f \) be a nonzero polynomial with complex coefficients in one variable such that \( f(x) = a \prod_{i} (x - \alpha_i) \). By Jensen's formula, the Mahler measure of $f$ can be expressed as:
\begin{equation}\label{mahler1}
m(f) = \log |a| + \sum_{i} \log \max(1, |\alpha_i|).
\end{equation}
If \( f \) has integer coefficients, is irreducible, and has a positive degree \( d \), then \eqref{mahler1} shows that \( m(f) \) is nonnegative, and

\[ h(\alpha) = \frac{m(f)}{d}, \]
where \( \alpha \) is a root of \( f \).
A classical theorem of Kronecker asserts that when  $f\in \Z[x]$, we have  $m(f)=0$  if and only if $|a|=1$ and 
$\alpha_i$ are zero or roots of the unity.    Boyd, Lawton, and Smyth independently showed (cf.~\cite{Boyd, Lawton-Kronecker, Smyth}) that the Mahler measure of an irreducible polynomial
\[
f \in \mathbb{Z}[x_1, \ldots, x_n],
\]
with \( f \neq \pm x_j \), is equal to $0$ if and only if {}{the zero set of \( f \) is a torsion variety}. \\ 

Let us point out that the case of hypersurfaces is particularly interesting.   The notion of normalized height is linked to the Mahler measure of one of its equations. More precisely,   let \( V \) be a hypersurface of \( \mathbb{G}_m^n \) defined over \( \mathbb{Q} \) and let \( f = 0 \) be one of its equations with integer coefficients of content 1. The \emph{normalized height} of \( V \) for the projective embedding that we have fixed, \( \mathbb{G}_m^n \hookrightarrow \mathbb{P}^n \), is nothing other than the Mahler measure of the polynomial \( f \): \[
\widehat{h}(V) = m(f).
\]
(cf.~\cite[Proposition 2.1(vii)]{DP}).\\

Lehmer asked whether there exists a constant $c>0$ such that for every polynomial
$f\in \Z[x]$ with $m(f)>0$, then $m(f)\geq c$. 

Lehmer's question can thus be translated into the following two equivalent conjectures:

\begin{conjecture}\label{Conjecture 1.1.} There exists a real number \(c > 0\) such that for every algebraic number \(\alpha \in \overline{\mathbb{Q}}^*\) of degree \(d\) over \(\mathbb{Q}\) that is not a root of unity, we have

\[ h(\alpha) \geq \frac{c}{d}. \]

\end{conjecture}

\begin{conjecture}\label{Conjecture 1.2.} There exists a real number \(c > 0\) such that for every irreducible polynomial \(f \in \mathbb{Z}[x]\), \(f \neq \pm x\), which is not a cyclotomic polynomial, we have
\[ m(f) \geq c. \]

\end{conjecture}

Within the framework of Conjectures \ref{Conjecture 1.1.} and \ref{Conjecture 1.2.}, the best known result to date is Dobrowolski's lower bound (see \cite{Dobrowolski}): There exists a constant $c$ such that for every \(\alpha \in \overline{\mathbb{Q}}^*\), an algebraic number of degree \(d\), and is not a root of unity, we have 

\begin{equation}\label{Dobrowolski}
h(\alpha) \geq \frac{c}{d} \left( \frac{\log(\log(3d))}{\log(3d)} \right)^3. 
\end{equation}
 This inequality can obviously be reformulated in polynomial language:

\[ m(f) \geq c \left( \frac{\log(\log(3d))}{\log(3d)} \right)^3 \]
for any irreducible polynomial \(f \in \mathbb{Z}[x]\), \(f \neq \pm x\), of degree \(d\), which is not a cyclotomic polynomial.\\

Analogues of Conjectures \ref{Conjecture 1.1.} and \ref{Conjecture 1.2.} have been considered in the context of elliptic curves (cf. \cite{David-elliptic, Hindry-Silverman-elliptic}). \\

Let us note that generalizations of Lehmer's problem to higher-dimensional varieties have been explored by Amoroso and David (cf.~\cite{AD1999, AD2000, AD2003}). For instance, they proposed the following conjectures.

\begin{conjecture}\label{mahler2.3}
For every positive  integer \(n \), there exists a constant \(c(n) > 0\) such that for any irreducible polynomial \(f \in \mathbb{Z}[x_1, \ldots, x_n]\), \(f \neq \pm x_j\), which is not a generalized cyclotomic polynomial, we have
\[ m(f) \geq c(n). \]
\end{conjecture}

Let $\boldsymbol{\alpha}=(\alpha_1,\ldots,\alpha_n)\in \mathbb{G}_m(\overline \Q)^n$.  Let \( F_0 \in \mathbb{Q}[x_1, \ldots, x_n] \) be a non-zero polynomial that vanishes at \( \boldsymbol{\alpha} \) and has minimal degree. Amoroso and David \cite{AD2000} defined  the \emph{obstruction index} of the point \( \boldsymbol{\alpha} \) as  the quantity \( \deg F_0 \), which they denote by \( \delta(\boldsymbol{\alpha}) \).  Amoroso and David noted that
\[1 \leq \delta(\boldsymbol{\alpha}) \leq nD^{1/n},
\]
where \( D = [\mathbb{Q}(\alpha_1, \ldots, \alpha_n) : \mathbb{Q}] \).\\

\begin{conjecture}\label{2.4} Let \( n \) be a positive  integer. There then exists a real number \( c(n) > 0 \) such that for all \( \boldsymbol{\alpha} = (\alpha_1, \ldots, \alpha_n) \in \mathbb{G}_m^n(\overline{\mathbb{Q}}) \), such that \( \alpha_1, \ldots, \alpha_n \) are multiplicatively independent, we have
\[ h(\boldsymbol{\alpha}) \geq \frac{c(n)}{\delta(\boldsymbol{\alpha})}. \]

\end{conjecture}

Amoroso and David \cite{AD1999} adapted the method of Dobrowolski in order to address Conjecture \ref{2.4}. Their work yielded the following result:
\[
h(\boldsymbol{\alpha}) \geqslant \frac{c(n)}{\delta(\boldsymbol{\alpha})} \log(3\delta(\boldsymbol{\alpha}))^{-\kappa(n)},
\]
where  $\kappa(n) = (n + 1)(n + 1)!^n - n$ (cf. \cite[Th\'eor\`eme 1.5]{AD1999}).

\section{Lehmer's problem for a Drinfel'd module}\label{sec3}

Let $p$ be a prime number and $n$ a positive integer. We put $q:=p^n$, and we denote by $\mathbb{F}_q$ the finite field of cardinal $q$. 
We denote by \( \mathbb{F}_q[T] \) the ring of polynomials in one variable with coefficients in the finite field \( \mathbb{F}_q \), by \( K = \mathbb{F}_q(T) \) its field of fractions, and by \( K_\infty = \mathbb{F}_q((1/T)) \) the completion of \( K \) with respect to the \( (1/T) \)-adic valuation \( v \). This valuation \( v \) is extended to an algebraic closure \( \overline{K} \) (respectively, \( \overline{K}_\infty \)) of \( K \) (respectively, \( K_\infty \)). We also denote by \( \iota \) the canonical homomorphism from \( \mathbb{F}_q[T] \) to \( \overline{K} \).\\

We have the usual Weil height on \(\mathbb{P}^N(\overline{K})\). If \(L\) is a finite extension of \(K\) of degree \(m\), the height of a point \(P = [x_0: \cdots:x_N]\) of \(\mathbb{P}^N(L)\) is given by:

\[ h(P) = \sum_m \frac{d(w)}{m} \max_{0\leq i\leq N} \left(-w(x_i)\right),  \]
where the sum is extended over all places of \(L\), \(d(w)\) denotes the residual degree at the place \(w\), normalized by \(w(L) = \mathbb{Z} \cup \{+\infty\}\). This height is independent of the field \(L\) chosen as shown by the usual properties of valuations (cf.~\cite{Lang-FDG}).

We define the height  $h(\alpha) $ of an element \(\alpha\) of \(\overline K\) as  the height $h([1:\alpha])$ of  the point \([1:\alpha]\). That is, we embed \(\overline K\) into \(\mathbb{P}^1\) and consider the Weil height relative to this embedding. \\

Let $\overline K\{\tau\}$ be the ring of twisted polynomials over $\overline K$. For every $f := \sum_{i=0}^{n} a_i \tau^i \in \overline K\{\tau\}$, we define its derivation $D$ at $0$ by $D(f) := a_0$. An homomorphism of $\mathbb{F}_q$-algebras $\phi : \mathbb{F}_q[T] \to \overline{K}\{\tau\}$ is called a \emph{Drinfel'd \ $\mathbb{F}_q[T]$-module} defined over $\overline{K}$ if $D \circ \phi = \iota$ and there exists $P \in \mathbb{F}_q[T]$ such that $\phi(P) \neq \iota(P) \tau^0$. According to \cite[Chap.~4, Lemma~4.5.1]{Goss}, there exists an integer $r > 0$ such that $\deg(\phi(P)) = r \deg(P)$, for all $P \in \mathbb{F}_q[T]$. The integer $r$ is called the \textit{rank} of $\phi$.\\

Denis \cite{Denis} defined the global height of the Drinfel'd module $\phi$ at $\alpha \in \overline{K}$ as follows:
\[
\hat{h}_\phi(\alpha) = \lim_{n \to \infty} \frac{h(\phi(P^n)(\alpha))}{q^{nr}},
\]
where $P$ is a polynomial in $\mathbb F_q[T]$ of positive degree $r$. This limit does not depend on the choice of $P$ (cf.~\cite[Th\'eor\`eme 1]{Denis}). Furthermore, Denis \cite[Th\'eor\`eme 2]{Denis} established an analogue of Dobrowolski's result \cite{Dobrowolski} when $\phi$ is a Carlitz module.

David and Pacheco \cite{David-Pacheco} solved the so-called abelian version of the Lehmer's problem  in the context of Drinfel'd modules. It is worth noting that the classical version of this problem, concerning non-torsion points in $\mathbb{G}_m(\mathbb{Q}^{\text{ab}})$, was solved by Amoroso and Dvornicich \cite{Amoroso-Dvornicich}. {}{The analogue of this result for abelian varieties was given by Baker and
Silverman, see \cite{Baker2003,BS2004,Silverman-lower}}.

\section{Adelic curves}\label{AC}

In this section, we briefly recall the theory of adelic curves. Let $K$ be a commutative ring and $M_K$ be the set of all absolute values on $K$. 
Following \cite{Huayi-Atsushi-Adelic}, an adelic structure on $K$ consists of a measure space $(\Omega,\mathcal A,\nu) $ equipped with a map $\phi:\omega\mapsto |\cdot|_\omega$ from 
$\Omega$ to $M_K$ satisfying the following properties:
\begin{enumerate}
\item $\mathcal A$ is a $\sigma$-algebra on $\Omega$ and $\nu$ is a measure on $(\Omega,\mathcal A)$;

\item for any $a\in K^\ast:=K\setminus \{0\} $, the function 
$\omega\mapsto \log|a|_\omega$ is $\mathcal A$-measurable, integrable with respect to
the measure $\nu$.

\end{enumerate}
This data is denoted by $S=(K,(\Omega,\mathcal A, \nu),\varphi)$.  \\

$S$ is said to be proper if the equality
\[
\int_\Omega \log|a|_\omega d\nu(\omega)=0,
\]
holds for each $a\in K^\ast$.\\

Number fields and function fields are fundamental examples of proper adelic curves 
(cf.~\cite[p. 170]{Huayi-Atsushi-Adelic}). Note that the properness of these curves corresponds to the product formula.\\

From now on, $S=(K,(\Omega,\mathcal A,\nu),\phi)$ is a proper adelic curve and $K^{ac}$ an algebraic 
closure of $K$. Let $S\otimes_K K^{ac}=({K^{\mathrm{ac}}},(\Omega_{K^{\mathrm{ac}}},\mathcal{A}_{K^{\mathrm{ac}}},\nu_{K^{\mathrm{ac}}}),\phi_{K^{\mathrm{ac}}})$ be the algebraic extension of $S$ by ${K^{\mathrm{ac}}}$ (see \cite[Paragraph 3.4.2]{Huayi-Atsushi-Adelic} for the construction).\\

\subsection{Height of points in $\p^n(K^{\mathrm{ac}} )$}\label{2.1}

For $[a_0:\ldots:a_n]\in \p^n({K^{\mathrm{ac}}})$,
the height of $[a_0:\ldots:a_n]$ with respect to the adelic curve $S$ is given as follows.
\[
h_S([a_0:\ldots:a_n])=\int_{\Omega^{ac}}\log \max (|a_0|_\chi,\ldots,|a_n|_\chi) \nu_{K^{\mathrm{ac}}}(d\chi),
\]
see \cite[Paragraph 3.5]{Huayi-Atsushi-Adelic} for more details about the construction.\\

For every $\alpha\in  K^{\textrm{ac}}$, we set
\begin{equation}\label{hS}
h_S(\alpha):=h_S([1:\alpha]).
\end{equation}

\begin{lemma} The following set
    \[
    \{ \alpha \in (K^{\mathrm{ac}})^\ast \mid h_S(\alpha)=0\},
    \]
    is a subgroup of $(K^{\mathrm{ac}})^\ast$.
    
\end{lemma}
\begin{proof} This follows from the following inequality
\[
 \max(1, xy)\leq \max(1,x)\max(1,y),\quad\forall x,y\geq 0.
 \]
 and the product formula (i.e. the properness of $S$).

\end{proof}

\subsection{The adelic structure on $\Q(T_1,\ldots,T_d)$}\label{QT}

Let $K=\Q(T_1,\ldots,T_d)$ be the field of rational functions of $d$ variables $T_1,\ldots,T_d$ and with coefficients in $\Q$. 
$\Q(T_1,\ldots,T_d)$ can be endowed with an interesting adelic structure, see \cite[Paragraph 3.2.6]{Huayi-Atsushi-Adelic}, or \cite[Section 3]{Moriwaki1}
and \cite[Example 2.5]{Burgos2016} for a more explicit construction. Let us review in detail
the case when $d=1$. 

Let $\Omega=\Omega_h\coprod \mathcal{P}\coprod [0,1]_\ast$, where $\Omega_h$ is the set of all closed points of $\p_\Q^1\setminus\{\infty\}$, and $\mathcal P$ denotes the set all prime numbers, and
$[0,1]_\ast$  the set of $t\in [0,1]$ such that $e^{2\pi it}$ is transcendental.

 Let $\phi:\Omega\rightarrow M_K$ be the map sending $\omega\in \Omega$ to $|\cdot|_\omega$, defined as follows. Let $x\in  \Omega_h$. {}{Then,} $x$ defines a discrete valuation on $K$ which is denoted by $\mathrm{ord}_x(\cdot)$. Let  $F_x$ be  an irreducible polynomial in $\Z[T]$ with $F_x(x)=0$. Let $H(x)$ be the Mahler measure of the polynomial $F_x$. It is defined as follows.

\[
H(x):=\exp\left(\int_0^1\log |F_x(e^{2\pi i t})|dt \right).
\]

Let $|\cdot|_x$ be the absolute value on $\Q(T) $ such that
\[
|\varphi|_x=H(x)^{- \mathrm{ord}_x(\varphi)},\quad \forall \varphi\in \Q(T).
\]
For $p$ a prime number, let $|\cdot|_p$ be $p$-adic value on $\Q$. It extends to $\Q[T]$ as follows.
\[
|\varphi|_p:=\max_{j=0,\ldots,d}|a_j|_p,\quad \forall \varphi=\sum_{i=0}^d a_i T^i\in \Q[T],
\]
and then, by multiplicativity, to $\Q(T)$.\\
 
For any $t\in [0,1]_\ast$, let $|\cdot|_t$ be the absolute value on $\Q(T)$ such that
\[
|\varphi|_t:=|\varphi(e^{2\pi i t})|,\quad \forall \varphi\in \Q(T),
\]
where $|\cdot|$ denotes the usual absolute value of $\C$. \\

$\Omega_h$ and $\mathcal P$ are equipped with the discrete $\sigma$-algebras, and $[0,1]_\ast$ 
with the restriction of the Borel $\sigma$-algebra on $[0,1]$. Let $\mathcal A$ be the $\sigma$-algebra on $\Omega$ generated by the above $\sigma$-algebras on $\Omega_h,\mathcal P$ and
$[0,1]_\ast$ respectively. Let $\nu$ be the measure on $\Omega$ such that
$\nu(\{x\})=1$ for $x\in \Omega_h$, that $\nu(\{p\})=1$ for any prime number $p$ and that the restriction of $\nu$ on $[0,1]_\ast$ coincides with the Lebesgue measure.

This data defines an adelic structure on $\Q(T)$. Moreover, $(\Q(T), (\Omega,\mathcal{A},\nu),\phi)$ 
is proper (cf.~\cite[Paragraph 3.2.5]{Huayi-Atsushi-Adelic}).

\subsection{On Lehmer's problem on adelic curves}

Motivated by the previous work on Lehmer's problem, it is natural to consider the following question:

\begin{question} Let $S=(K,(\Omega,\mathcal A,\nu),\phi)$ be a proper adelic curve and $K^{ac}$ an algebraic 
closure of $K$.  Is there a constant $c:=c_S>0$  such for every $\alpha\in K^{\mathrm{ac}}$ with $h_S(\alpha)>0$,  we have
\[
h_S(\alpha)\geq \frac{c}{[K(\alpha):K]}?
\]

\end{question}

When $K$ is $\Q$ and $S$ is the classical adelic curve associated with $\Q$, then this is the classical Lehmer's problem.

\section{Lehmer's problem on $\overline{\Q(T)}$}\label{main}
In this section, we  show that the classical Lehmer's problem generalizes naturally to the adelic curve $\Q(T)$.\\

Let $S$ be the adelic curve on $\Q(T)$ recalled in Paragraph \ref{QT}. Let 
\[
[\varphi:\psi]\in \p^1\left(\overline{\Q(T)} \right).
\]

We may assume that $\varphi$ and $\psi$ are integral over $\Z[T]$. Then, there exist 
$d,e\in \N$ and $f_1(T),\ldots f_d(T)$ and $g_1(T),\ldots,g_e(T)$  in $\Z[T]$ such that
\[
\varphi^d+f_1(T)\varphi^{d-1}+\ldots+f_{d-1}(T)\varphi+f_d(T)=0,
\]
and
\[
\psi^{e}+g_1(T)\varphi^{e-1}+\ldots+g_{e-1}(T)\psi+g_e(T)=0.
\]

Let 
\[
c_{\varphi,\psi}:=\gcd(c(f_d)^{e}, c(g_e)^{d}).
\]
where $c(f_d)$ (resp. $c(g_e)$) is the common divisor of the coefficients of $f_d$ (resp. $g_e$). \\

If $\{x\in \overline \Q\mid g_e(x)=f_d(x)=0\} $ is not empty, we let 
\[
F(T):=\prod_{\substack{x\in \overline \Q \\
g_e(x)=f_d(x)=0}} F_x(T)^{\min(\mathrm{ord}_x(f_d^e),\mathrm{ord}_x(g_e^d))}\,,
\]
with $F_x\in \Z[T]$ with coprime coefficients, irreducible in $\Q[T]$ and $F_x(x)=0$ for $x\in \overline \Q$. Otherwise, we let
\[
F(T)=1.
\]

By the product formula, we have 
\[
h_S([\varphi:\psi])=h_S\Bigl(\Bigl[\tfrac{\varphi}{(c_{\varphi,\psi} F)^{\frac{1}{de}}}:\tfrac{\psi}{(c_{\varphi,\psi} F)^{\frac{1}{de}}}\Bigr]\Bigr).
\]

Let $K'$ be a finite extension of $\Q(T)$ such that $\varphi$ and $\psi$ are in $K'$.\\

Let $\omega\in \Omega_{K'}$ with $\omega\mid x$ where $x\in \Omega_h$. If $x$ is a common zero of $f_e$ and $g_d$, then
\[
\begin{split}
\log \left|\frac{\varphi}{(c_{\varphi,\psi} F)^{\frac{1}{de}}}\right|_\omega=&\frac{1}{d}\log |f_d|_x-\frac{1}{de}\log |F|_x\\
=& -\frac{\mathrm{ord}_x(f_d(T)) }{d}\log H(x)+\frac{ \min(\mathrm{ord}_x(f_d(T)^e), \mathrm{ord}_x(g_e(T)^d))}{de}\log H(x).
\end{split}
\]
and 
\[
\begin{split}
\log \left|\frac{\psi}{(c_{\varphi,\psi} F)^{\frac{1}{d}}}\right|_\omega=&\frac{1}{e}\log |g_e|_x-\frac{1}{de}\log |F|_x\\
=& -\frac{\mathrm{ord}_x(g_e(T))}{e}\log H(x)+\frac{ \min(\mathrm{ord}_x(f_d(T)^e), \mathrm{ord}_x(g_e(T)^d) )}{de}\log H(x).
\end{split}
\]

We conclude that
\[
\max\left(\left|\frac{\varphi}{(c_{\varphi,\psi} F)^{\frac{1}{de}}}\right|_\omega,\left|\frac{\psi}{(c_{\varphi,\psi} F)^\frac{1}{de}} \right|_\omega\right)=1
\]
This equation holds if $x$ is not a common zero of $f_e$ and $g_d$.\\

Let $\omega \in \Omega$ with $\omega\mid p$ and $p\in \mathcal{P}$. We have
\[
\left|\frac{\varphi}{  (c_{\varphi,\psi} F) ^{\frac{1}{de}}}\right|_\omega=\frac{|f_d(T)|_p^{\frac{1}{d}}}{| (c_{\varphi,\psi} ) ^{\frac{1}{de}} |_p^{}}=\left(\frac{|c(f_d)^{e }|_p^{}} {|c_{\varphi,\psi} |_p}  \right)^{\frac{1}{de}},
\]
and
\[
\left|\frac{\psi}{(c_{\varphi,\psi} F) ^{\frac{1}{de}}}\right|_\omega=\frac{|g_e(T)|_p^{\frac{1}{e}}}{|(c_{\varphi,\psi} ) ^{\frac{1}{de}}|_p^{ }}=\left(\frac{|c(g_e)^{d }|_p^{}} {|c_{\varphi,\psi} |_p}  \right)^{\frac{1}{de}}.
\]

We conclude that
\[
\max\left(\left|\frac{\varphi}{(c_{\varphi,\psi} F)^{\frac{1}{de}}}\right|_\omega,\left|\frac{\psi}{(c_{\varphi,\psi} F)^\frac{1}{de}} \right|_\omega\right)=1.
\]

Case of $w\in \Omega$ with $w\mid t$ and $t\in [0,1]_\ast$.  
Note that $K'_\omega=\Q(T)_t=\C$. We have then
\[
\begin{split}
\int_{\substack{w\in \Omega\\
w\mid t,\ t\in [0,1]_\ast }}  \log \max\Bigl( &\Bigl  |\frac{\varphi}{(c_{\varphi,\psi} F)^{\frac{1}{de}}}\Bigr|_\omega, 
\Bigl |\frac{\psi}{(c_{\varphi,\psi} F)^\frac{1}{de}}  
\Bigr|_\omega\Bigr) d\nu(\omega)\\=
&\int_{t\in [0,1]_\ast} \frac{1}{[K':\Q(T)]} \left( \sum_{\substack{ w\in K'\\ w\mid t}}  \log \max\Bigl( \Bigl  |\frac{\varphi}{(c_{\varphi,\psi} F)^{\frac{1}{de}}}\Bigr|_\omega, 
\Bigl |\frac{\psi}{(c_{\varphi,\psi} F)^\frac{1}{de}}  
\Bigr|_\omega\Bigr) \right)dt.
\end{split}
\]
Gathering the  above computations, we obtain that 
\[
\frac{1}{[K':\Q(T)]}\max\Bigl(\int_{t\in [0,1]_\ast}  \sum_{\substack{ w\in K'\\ w\mid t}}  \log  \Bigl  |\frac{\varphi}{(c_{\varphi,\psi} F)^{\frac{1}{de}}}\Bigr|_\omega dt,  
\int_{t\in [0,1]_\ast}  \sum_{\substack{ w\in K'\\ w\mid t}}  \log  \Bigl  |\frac{\psi}{(c_{\varphi,\psi} F)^{\frac{1}{de}}}\Bigr|_\omega dt\Bigr) \leq h_S([\varphi:\psi]).
\]

Using \cite[Lemma 3.5.4]{Huayi-Atsushi-Adelic}, we obtain 
{\small{
\[
\begin{split}
\tfrac{1}{[K':\Q(T)]}\max\Bigl(\int_{t\in [0,1]_\ast}  \sum_{\substack{ w\in K'\\ w\mid t}}  \log  \Bigl  |\mathrm{N}_{K'/\Q(T)}\left(\tfrac{\varphi}{(c_{\varphi,\psi} F)^{\frac{1}{de}}}\right) \Bigr|_\omega dt ,  
\int_{t\in [0,1]_\ast}  \sum_{\substack{ w\in K'\\ w\mid t}} & \log  \Bigl  |\mathrm{N}_{K'/\Q(T)}\left(\tfrac{\psi}{(c_{\varphi,\psi} F)^{\frac{1}{de}}}\right) \Bigr|_\omega dt \Bigr)\\  
&\leq [K':\Q(T)] h_S([\varphi:\psi]).
\end{split}
\]
}}
Hence
\begin{equation}\label{key2}
\begin{split}
\max\Bigl(\int_{t\in [0,1]}   \log  \Bigl  |\mathrm{N}_{K'/\Q(T)}\left(\tfrac{\varphi}{(c_{\varphi,\psi} F)^{\frac{1}{de}}}\right) \Bigr|_t dt,  
\int_{t\in [0,1]}   & \log  \Bigl  |\mathrm{N}_{K'/\Q(T)}\left(\tfrac{\psi}{(c_{\varphi,\psi} F)^{\frac{1}{de}}}\right) \Bigr|_t dt\Bigr)\\  
&\leq [K':\Q(T)] h_S([\varphi:\psi]).
\end{split}
\end{equation}

\begin{theorem}[A Kronecker's theorem]\label{Kronecker}

\[
 \{ \alpha \in (\overline{\Q(T)})^\ast \mid h_S(\alpha)=0\}=\{\text{roots of unity in} \ \overline{\Q}\}.
\]

\end{theorem}

\begin{proof}

Let $\varphi$ and $\psi$ be two coprime polynomials in $ \Z[T]$. We have
\[
h_S([\varphi:\psi])=\int_0^1\log \max(|\varphi(e^{2\pi i t})|, |\psi(e^{2\pi i t})|)dt.
\]
Let us assume that  $h_S([\varphi:\psi])=0$. This implies that
\[
\varphi(z), \psi(z),\ \text{and}\ \varphi(z)+y \psi(z)
\]
are generalized cyclotomic polynomials in the variables $y,z$ (see for instance \cite{Boyd}).  We have used the following fact:
\[
\int_0^1\log \max(|\varphi(e^{2\pi i t})|, |\psi(e^{2\pi i t})|)dt=\int_{0}^1 \int_0^1 \log|\varphi(e^{2\pi i t})+e^{2\pi  is}\psi(e^{2\pi i t}) | dtds.
\]
It is not difficult to show that
$\phi$ and $\psi$  are, in fact, roots of unity in $\overline \Q$.\\

Let $\alpha\in \overline{\Q(T)}$ with $h_S(\alpha)=0$.  Using \eqref{key2}, we can conclude that $\alpha$ is a root of unity in $\overline \Q$.

\end{proof}

 Let $\varphi\in \overline{\Q(T)}$. Let $d:=[\Q(T)(\varphi):\Q(T)]$. We can assume
\[
[1:\varphi]=[\varphi_1:\varphi_2],
\]
with $\varphi_1\in \Z[T]$ and $\varphi_2\in \overline{\Z[T]}$ (integral closure of $\Z[T]$).\\ 

From \eqref{key2} and assuming that Lehmer's problem is true, we get 
\[
h_S(\varphi)(=h_S([\varphi_1:\varphi_2]))\geq  \frac{c}{\deg(\varphi)},
\]
where $\deg \varphi=[\Q(T)(\alpha):\Q(T)]$ and $c$ is the constant in Lehmer's problem. \\

\begin{remark}\label{rmk5.2}

Let $\alpha \in \overline{\mathbb{Q}}$. It is clear that

\[
h_S(\alpha) = h(\alpha),
\]
where $h$ is the standard Weil height of $\alpha$.
\end{remark}

We conclude that
\begin{theorem}\label{Lehmer}
Lehmer's problem is equivalent to Lehmer's problem on $\Q(T)$.
\end{theorem}

\bigskip 

\noindent\textbf{Acknowledgement:} I would like to thank the referee for their insightful comments and suggestions.

\bibliographystyle{plain}

\bibliography{biblio}

\end{document}